\documentclass[12pt]{amsart}

\usepackage{amsmath,amssymb,amsthm}
\usepackage[dvipdfm]{graphicx}
\usepackage{enumerate}
\usepackage[dvips]{color}
\usepackage{version}

\newtheorem{Thm}{Theorem}[section]

\newtheorem{Prop}[Thm]{Proposition}
\newtheorem{Prob}[Thm]{Problem}
\newtheorem{Conj}[Thm]{Conjecture}
\newtheorem{``Conj"}[Thm]{``Conjecture"}

\theoremstyle{remark}

\theoremstyle{definition}
\newtheorem{Def}[Thm]{Definition}

\begin{document}

\title[Numerical invariants after K-stability]{Invariants of varieties and singularities 
inspired by K\"ahler-Einstein problems}
\dedicatory{Dedicated to Professor Futaki's sixtieth birthday \\ 
and retirement of Professor Mabuchi, \\ 
in honor of their pioneering works} 

\author{Yuji Odaka}
\address{Department of Mathematics, Kyoto University,
Kyoto, Japan}
\email{yodaka@math.kyoto-u.ac.jp}


\maketitle

\begin{abstract}
We extend the framework of K-stability \cite{Tia}, \cite{Don} 
to more general algebro-geometric setting, such as partial desingularisations of 
(fixed) singularities, (not necessarily flat) families over higher dimensional 
base and the classical birational geometry of surfaces. 

We also observe that ``concavity'' of the volume function 
implies decrease of the (generalised) Donaldson-Futaki invariants along the Minimal 
Model Program, in our generalised settings. 
Several related results on the connection with the MMP theory, 
some of which are new even in the original setting of families over curves, 
are also proved. 
\end{abstract}


\section{Introduction}

Working on canonical (K\"{a}hler) metrics via the use of \textit{numerical 
invariants} has its origin in the seminal and pioneering paper of A.~Futaki 
\cite{Fut}. 
Shortly after the introduction, 
T.~Mabuchi \cite{Mab} introduced important functionals over 
the space of K\"{a}hler potentials which are 
connected to the Futaki's invariant. 
Those two fundamental and pioneering works, 
in turn, after one decade passing, 
revealed their algebraic natures as shadows of a variant of the Mumford's GIT stability 
- K-stability formulated by Tian and Donaldson \cite{Tia}, 
\cite{Don}. It is also compatible with Yau's original insightful expectation 
\cite{Yau} that some ``stability'' should be equivalent to the existence of 
canonical metrics. 

This paper aims at clarifying and generalising those invariants in 
more general algebro-geometric setting, which are not necessarily for 
families over a curve, such as (partial) desingularisations of singularities, 
families over higherdimensional base or classical (absolute) birational geometry of 
surfaces. We also add results which are new even in the 
original setting of families over curves. The author also hopes that 
this would serve as a suplementary introduction for algebraic geometers 
to the subject. 















\section{Generalised Setting}\label{Gen.set}

From now on, for simplicity, we fix an arbitrary algebraically closed field of 
characteristic $0$ on which we work. The arguments which are not depending on 
the resolution of singularities nor the Minimal Model Program (MMP), 
work also over positive characteristics fields. 

We fix a normal $\mathbb{Q}$-Gorenstein projective variety $B$ as a base.  
We also fix a base point $p\in B$ and a projective morphism 
$\pi^{o}\colon (\mathcal{X}^{o},\mathcal{L}^{o})\twoheadrightarrow B\setminus \{p\}$ 
where $\mathcal{L}^{o}$ are $\pi^{o}$-ample and 
$\mathcal{X}^{0}$ is also normal $\mathbb{Q}$-Gorenstein projective variety 
of dimension $n$. Many of our theory extends naturally also to 
non-normal varieties which are reduced, equidimensional algebraic 
schemes, $\mathbb{Q}$-Gorenstein, 
Gorenstein in codimension $1$, and satisfying the Serre's condition $S_{2}$. 

We consider all the completions of $(\mathcal{X}^{o},\mathcal{L}^{o})$ to 
over $B$ i.e. projective morphisms 
$\pi\colon(\mathcal{X},\mathcal{L})\twoheadrightarrow B$ 
with \textit{$\pi$-nef} $\mathcal{L}$ such that $\pi^{-1}(B\setminus \{p\})
=(\mathcal{X}^{o},\mathcal{L}^{o})$. 
In other words, we consider birational modifications along $\pi^{-1}(p)$. 

Original setting after Mumford, Futaki, Tian, Donaldson 
is that 
$(\mathcal{X}^{0},\mathcal{L}^{o})=(X,L)\times 
(B\setminus\{p\})$ with $B=\mathbb{P}^{1}$ (or $\mathbb{A}^{1}$ originally), 
$(X,L)$ is polarized projective variety, 
and $(\mathcal{X},\mathcal{L})$ 
is ``test configuration'' (\cite{Don}). 
Hence, our main point of the extension is that we allow all kinds of projective morphism, for example, $\pi$ can be non-flat or even birational.

To our general $\pi\colon(\mathcal{X},\mathcal{L})\twoheadrightarrow B$, 
we assign the following two invariants ${\it DF}$ and 
$\mathcal{V}$ whose K\"ahler analogues appeared in 
\cite{Fut, Tia, Don} and \cite{Aub, Mab}. 

\subsection{Normalised volume functional}

The first invariant we introduce is the following $\mathcal{V}$. 
It will become clear only later why this simple but somehow modified 
version of volume function is important for us. 

\begin{Def}
$\mathcal{V}(\mathcal{X},\mathcal{L}):=
\frac{(\mathcal{L}^{n})}
{(\mathcal{L}^{o}|_{F}^{{\it dim}(F)})^{\frac{n-1}{{\it dim}(F)}}},$

\noindent
where $F$ is the generic fiber of $\pi^{0}$, 
which people often denote by $\mathcal{X}_{\eta}$. 
If ${\it dim}(F)$=0, i.e. $\pi^{0}$ is generically finite, 
then, we interpret the denominator of the above as $1$. 
The meaning of a little complicated denominator as a normalizing factor, 
which indeed only depends on the original $\mathcal{L}^{o}$, will be clear in the proof of 
Theorem \ref{Fut.decrease}. 

\end{Def}

We call the above 
\textit{normalized volume functional}. 
For the case with ${\it dim}(B)=1$, 
this may be able to seen as a functional of non-archimedian 
smooth semipositive metrics from the perspective of \cite{BFJ}, 
which is concave (cf. next proposition \ref{Mabuchi.basic} (iii)). 
Professor S.~Boucksom kindly pointed out to the author that 
${\it dim}(B)=1$ case of the above essentially coincides with 
the non-archimedian analogue of 
\textit{Aubin-Mabuchi functional} (also called ``Monge-Amp\'{e}re energy'') 
they discussed in \cite{BFJ2}. Please note it is different from the 
so-called K-energy of \cite{Mab}. 

\begin{Prop}[Basic properties of $\mathcal{V}$]\label{Mabuchi.basic}

Regarding the above normalised volume functional $\mathcal{V}$, 
we have the following basic properties: 

\begin{enumerate}

\item 

Consider the pull back of $\mathcal{L}$ by a 
birational morphism $f\colon \mathcal{X}'\rightarrow \mathcal{X}$. 
$$
\mathcal{V}(\mathcal{X}',f^{*}\mathcal{L})=\mathcal{V}(\mathcal{X},\mathcal{L}). 
$$

\item If ${\it dim}(F)>0$, then 
the functional is homogeneous of degree $1$ i.e., 
for any $a\in\mathbb{Z}_{>0}$, 

$$\mathcal{V}(\mathcal{X},\mathcal{L}^{\otimes a})=
a\cdot\mathcal{V}(\mathcal{X},\mathcal{L}). $$




 


\item

For any Cartier divisor $E$ supported on a fiber, 

$$\frac{\partial^{2}}{\partial{E}^{2}}\mathcal{V}
(\mathcal{X},\mathcal{L})\leq 0, $$

\noindent
that is, this functional $\mathcal{V}$ is concave along the space of 
divisors supported on $\pi^{-1}(p)$. 

\end{enumerate}

\end{Prop}

\noindent
$(iii)$ may be able to be regarded as algebraic 
version of the convexity of (differential geometric) Aubin-Mabuchi functional. 

\begin{proof}

$(i)$, $(ii)$ are straightforward to see so we omit the proofs. 

We prove that $(iii)$ essentially follows from the Hodge index theorem. 
Also, here is the place we use the assumption that birational modifications are 
all along fibres over finite points of $B$. 
Let us take an arbitrary ample divisor $H$ on $B$. 

Consider variation of $\mathcal{L}$ to $\mathcal{L}(tD)$ where $D$ 
is a Cartier divisor supported on $\pi^{-1}(p)$. What we need to prove is 

$$\frac{\partial^{2}}{\partial D^{2}}(\mathcal{L}^{n})=n(n-1)
(\mathcal{L}^{n-2}.D^{2})\leq 0.$$

The above follows from the Hodge index theorem since 
$$
(\mathcal{L}^{n-2}.D.\pi^{*}H)=0, 
$$
because $\pi({\it Supp}(D))$ is zero-dimensional. 
An important note is that our $(iii)$ above is an analogue of the 
concavity of the Aubin-Mabuchi functional (cf., e.g. \cite{BBGZ}). 
\end{proof}

We note that via $(ii)$ of the above proposition, for ${\it dim}(B)=1$ case, 
we can regard $\mathcal{V}$ as a functional over a 
space of all $\mathbb{R}$-line bundles up to pull back 
(``\textit{infinite dimensional nef cone}''). In other words, the space consists of 
smooth non-archimedian smooth semipositive 
metrics on the analytification $\mathcal{L}^{an}$ 
of $\mathcal{X}^{an}$ from the viewpoint of \cite{BFJ} (cf. also \cite{KT}). 
I would like to thank Sebastien Boucksom for teaching me about the non-archimedian 
metrics they use. 




\subsection{Generalising Futaki's invariants further}

Now we introduce our second invariant - 
an extension of the Futaki invariants, generalising the Donaldson's 
extension of the Futaki invariant \cite{Don} further. 
We begin with the following observation for 
original setting  i.e. families over curves. 

Our definition is ``derivative'' along the direction of the 
(relative) canonical divisor, 
which is exactly encoding the infinitisimal behaviour along the Minimal Model 
Program with scaling \cite{BCHM} (or its analytic counterpart, i.e. 
unnormalised K\"ahler-Ricci flow 
cf., e.g., \cite{CL}, \cite{ST}).

\begin{Prop}[{\cite{Wan},\cite{Od1}}]

If $(\mathcal{X},\mathcal{L})$ is a test configuration of a 
polarised projective variety $(X,L)$, we have 

$$\frac{\partial}{\partial {K_{\mathcal{X}/B}}}\mathcal{V}(\mathcal{L})=
\frac{{\it DF}(\mathcal{X},\mathcal{L})}{(L^{n-1})^{2}},$$ 

\noindent
where the ${\it DF}(\mathcal{X},\mathcal{L})$ is the Donaldson-Futaki invariant 
\cite{Don}. In general, if $B$ is a curve, then the above equation holds once 
we replace the Donaldson-Futaki number by 
$
{\it deg}(\lambda_{{\it CM}}(\mathcal{X},\mathcal{L})), 
$
\noindent
where the $\lambda_{{\it CM}}$ is the CM line bundle introduced by 
Paul-Tian \cite{PT} (also see \cite{FR}). 
\end{Prop}

Given the above, we define our further 
generalisation of the Donaldson-Futaki invariants 
similarly as follows. 

\begin{Def}

For our extended framework of the previous section, 
we define our (generalised) Donaldson-Futaki invariant as 

$${\it DF}(\mathcal{X},\mathcal{L}):=
\biggl(\frac{\partial}{\partial {K_{\mathcal{X}/B}}}\mathcal{V}(\mathcal{L})\biggr)
\cdot{
{(\mathcal{L}^{o}|_{F}^{{\it dim}(F)})^{\frac{2(n-1)}{{\it dim}(F)}}}}. 
$$

More explicitly, we can write ${\it DF}(\mathcal{X},\mathcal{L})$ as follows: 

$$
n(\mathcal{L}^{n-1}.K_{\mathcal{X}/B})(\mathcal{L}|_{F}^{{\it dim}(F)})
^{\frac{n-1}{{\it dim}(F)}}$$
$$
-(n-1)(\mathcal{L}^{n})
(\mathcal{L}|_{F}^{{\it dim}(F)-1}.K_{F})
(\mathcal{L}|_{F}^{{\it dim}(F)})^{\frac{{\it dim}(B)-1}{{\it dim}(F)}}. 
$$

\noindent
Note that the last term $(\mathcal{L}|_{F}^{{\it dim}(F)})^{\frac{{\it dim}(B)-1}
{{\it dim}(F)}}$ has not appeared in the original setting of ${\it dim}(B)=1$ 
since the exponent was $0$ in that case. 

\end{Def}

The following basic property says that Donaldson-Futaki invariants 
are essentially a functional of polarisations (up to pull backs).

\begin{Prop}\label{DF.pullback}

Consider the pull back of $\mathcal{L}$ on $\mathcal{X}$ by a 
birational morphism $f\colon \mathcal{X}'\rightarrow \mathcal{X}$. Then we have 
$$
DF(\mathcal{X}',f^{*}\mathcal{L})= DF(\mathcal{X},\mathcal{L}). 
$$

\end{Prop}

The proof follows straightforward from the above description of 
our (generalised) Donaldson-Futaki invariants via 
intersection numbers. 
After this proposition, we often write the above generalisation simply 
as ${\it DF}(\mathcal{L})$ and call DF invariant or DF (of the polarisations) 
from now on in this paper. 

Note that the above (\ref{DF.pullback}) would be inequality $\ge$ in general if $\mathcal{X}$ is non-normal, 
whose difference in turn reflects the presence of conductor of normalisation. 
This extends the old result of Ross-Thomas \cite[(5.1),(5.2)]{RT} and again also 
matches to the non-archimedian framework of \cite{BFJ}. 

We can extend our Futaki invariant further to log setting by using 
Shokurov's ``b-divisors'' \cite{Sho}, which is roughly speaking 
infinite linear combination of divisors above $\mathcal{X}$. 
Indeed, $\{K_{\mathcal{X}/B}\}_{\mathcal{X}}$ 
also form a Weil b-divisor and 
it is enough to replace it by $\{K_{\mathcal{X}/B}\}_{\mathcal{X}}+
\mathcal{D}$ where 
$\mathcal{D}$ is some other Weil b-divisor. Accordingly, all the following 
contents of our theory extend in straightforward manner but we omit them. 
The extension includes the usual log K-stability \cite{Don.l}, \cite{OS}.

\section{Extending K-semistability}

We extend the idea of K-stability \cite{Don} 
to our generalised framework. More precisely, 
we extend K-semistability and study its property as follows. 

\begin{Def}
We follow the notation of the previous section. 
We call $\pi\colon (\mathcal{X},\mathcal{L})\twoheadrightarrow B$ is 
\textit{generically K-semistable} if the set of all the 
Donaldson-Futaki invariants of 
all possible birational transforms along $\pi$-preimages of finite points 
are lower bounded. 

\end{Def}

Note that if $\pi$ is a test configuration, the above implies the original 
K-semistability of general fibers. 
However, in general, our definition is not the 
same as K-semistability of generic fiber as we will see 
soon in Proposition \ref{desing.ss}. 
The above is a little analogous to the fact that 
original K-semistability for \cite{Tia} and \cite{Don} 
corresponds to lower boundedness of Mabuchi functional. 
From this section, even including the above definition, 
we will observe that the Futaki invariant itself shares some feature of the 
Mabuchi functional (K-energy) \cite{Mab}. 

\begin{Prop}\label{desing.ss}
If $\mathcal{X}^{o}\rightarrow B\setminus\{p\}$ is an isomorphism, 
so that any completion of $\pi^{o}$ is birational, 
then it is generically K-semistable if and only if $B$ has only canonical 
singularities. Moreover, all the non-trivial Donaldson-Futaki invariants are 
positive if and only if $B$ has only terminal singularities. 
\end{Prop}

\noindent
We may call the latter as ``generically K-stable'' but in general setting, 
we do not know how to formulate ``(non-)triviality'' of the completion 
$(\mathcal{X},\mathcal{L})$. 

\begin{proof}[proof of Proposition \ref{desing.ss}]
Suppose $B$ is not canonical at $p\in B$. 
Then we take a relative canonical model $\mathcal{X}^{{\it can}}$ over 
$B$ by \cite{BCHM}. By the negativity lemma \cite[3.39]{KM}, 
it follows that $K_{\mathcal{X}^{{\it can}}/B}$ is anti-effective (and non-zero). 
We take 
$\mathcal{L}^{{\it can}}:=K_{\mathcal{X}^{{\it can}}/B}$ which is $\pi$-ample 
from our construction. Thus, the Donaldson-Futaki invariant 
${\it DF}(\mathcal{X}_{{\it can}},\mathcal{L}^{{\it can}})$, 
which is $(\mathcal{L}_{{\it can}}^{n-1}.K_{\mathcal{X}_{{\it can}}/B})$ up to 
positive multiple constant, is negative. Thus we end the proof of the former 
half of the assertions. 

If $B$ is strictly canonical at $p$, i.e. canonical but not terminal, 
take terminalisation $\mathcal{X}^{{\it min}}$ 
of $B$ again by \cite{BCHM}. Then since we know 
$K_{\mathcal{X}_{{\it min}}/B}=0$, for an arbitrary 
$\pi$-nef line bundles $\mathcal{L}$ on $\mathcal{X}_{{\it min}}$, 
the corresponding Donaldson-Futaki invariants vanish. 
\end{proof}

For general fibration case, 
we can also see that 

\begin{Prop}
If $K_{F}=a\mathcal{L}|_{F}$ with $a\ge 0$, 
and $\mathcal{X}\setminus \pi^{-1}(p)$ has only canonical singularities 
with ${\it dim}(B)=1$, then 
$(\mathcal{X},\mathcal{L})$ is generically K-semistable. 
\end{Prop}

\begin{proof}
We give a case by case proof depending on the 
Fujita-Kawamata type semipositivity (cf., e.g, \cite{Kol},\cite{Fuj}). 

If $a=0$, this is a Calabi-Yau fibration. From the semipositivity theorem, 
we can take $K_{\mathcal{X}/B}$ as an effective vertical divisors. 
Therefore, the 
corresponding Donaldson-Futaki invariant 
$(\mathcal{L}^{n-1}.K_{\mathcal{X}/B})$ multiplied by some positive constant, 
is non-negative. 

If $a>0$, here we only prove the case when 
$K_{\mathcal{X}/B}$ is relatively ample and $\mathcal{X}$ is canonical, 
i.e. the relative canonical model over $B$ and leave the rest to 
Theorem \ref{DF.decrease}. In that case, 
the corresponding Donaldson-Futaki invariant is simply 
$(K_{\mathcal{X}/B}^{n})$ up to positive multiple constant. 
Then it is the leading coefficient of 
$${\it deg}_{B}({\it det}(\pi_{*}\mathcal{O}_{\mathcal{X}}
(mK_{\mathcal{X}/B}))),$$
thanks to the Grothendieck-Riemann-Roch theorem. 
The semipositivity theorem \cite{Kol},\cite{Fuj} implies 
that the above quantity is all non-negative for sufficiently divisible 
$m\in \mathbb{Z}_{>0}$, so the assertion holds for the relative canonical model 
case. For general case, from 
Theorem \ref{DF.decrease} shows that birational modifications of 
this relative canonical model has bigger or equal Donaldson-Futaki invariants 
so that the assertion holds. 
\end{proof}

Motivated by the above, next section \ref{Fut.small}, 
and the relation of CM line bundle with the Weil-Petersson metrics, 
it is natural to conjecture that 

\begin{Conj}
If ${\it dim}(B)=1$, 
the generic K-semistability of $(\mathcal{X},\mathcal{L})$ is 
equivalent to that generic fiber $(F,\mathcal{L}|_{F})$ 
of $\pi$ is K-semistable in the original sense \cite{Don}. 
\end{Conj}

A similar inspiring conjecture was asked before by X.Wang during my visit to Hong Kong 
in spring of 2010, and would like to express thanks to the hospitality of him and 
N-C.Leung. 


\section{To minimise the DF by the MMP}\label{Fut.small}

\subsection{Decrease of the DF along MMP}

Recall that \cite{Od0} observes, very vaguely speaking, that birationally \textit{``small'' models} 
in the MMP theory have ``small'' Donaldson-Futaki 
invariants. A while after \cite{Od0}, 
the \textit{continuous decrease} 
of the Donaldson-Futaki invariants along the MMP was first 
proved in \cite{LX} for families of Fano varieties over curves. 
As a K\"ahler analogue, this should corresponds to decrease of K-energy along 
the normalised K\"ahler-Ricci flow. 

We generalise the phenomenon to our much extended framework as follows. 

\begin{Thm}\label{Fut.decrease}
Suppose $K_{F}=a\mathcal{L}|_{F}$ and let us consider 
the $K_{\mathcal{X}/B}$-MMP with ``scaling $\mathcal{L}$'' 
(precisely speaking, the scaling divisor is $\frac{1}{l}\mathcal{D}$ 
where $\mathcal{D}$ corresponds to a general section of 
$(\mathcal{L}\otimes \pi^{*}\mathcal{M})^{\otimes l}$ 
with sufficiently ample $M$ on $B$ and sufficiently big $l\in \mathbb{Z}_{>0}$. 

Along that MMP with scaling, which we see as the linear and continuous change 
of polarisations $\mathcal{L}_{t}$ to the relative canonical divisor, 
the Donaldson-Futaki invariants ${\it DF}(\mathcal{L}_{t})_{t\ge 0}$ 
monotonely decrease when $t$ increases. 
\end{Thm}

\begin{proof}
We have $\mathcal{L}_{t}=tE+\mathcal{L}$ 
where 
$E:=K_{\mathcal{X}/B}-a\mathcal{L}$. 
Thus what we need to show is for the direction to $E$, 
the Futaki decreases i.e., 
$
\frac{\partial}{\partial E}({\it DF}(\mathcal{L}))<0. 
$
It follows from the properties of Mabuchi functional $\mathcal{V}$ 
(\ref{Mabuchi.basic}). 
Indeed, after some simple calculation using (\ref{Mabuchi.basic}(ii)), 
this derivative can be rewritten via $\mathcal{V}$ as 
$\frac{\partial^{2}}{\partial E^{2}}\mathcal{V}$ which is 
negative by the convexity of $\mathcal{V}$ (\ref{Mabuchi.basic}(iii)). 

We would like to leave the check of the simple calculation to the readers. 
\end{proof}

Note it is proved that K\"{a}hler-Ricci flow 
is 
compatible with (K-)MMP with ample scaling (\cite{CL},\cite{ST} etc). 
This compatibility can be also seen when we transfer the K\"ahler-Ricci flow 
into non-archimedian setting. It would be interesting to see the above 
phenomenon from differential geometric Ricci flow point of view. 

\subsection{Minimisation -semistable case-}

In our algebraic setting, the observations below show that 
\textit{minimisation (critical points) of our 
generalised Donaldson-Futaki invariants} give a 
``canonical limits'' of ``semistable'' objects, 
which we see as an analogue of the fact that 
critical points of the K-energy \cite{Mab} are those with 
constant scalar curvature (e.g. K\"ahler-Einstein metrics). 

These phenomenons are motivated by \cite{Od0} and \cite{LX}. 
Indeed the method of \cite{Od0} to get small (in that case, negative) 
Futaki invariants, taking some ``canonical model'' in the MMP theory, 
can be interpretted as a phenomenon that 
``\textit{canonical} limits (made by the MMP) give \textit{small} Futaki invariants''. 
Furthermore, \cite{LX} later proved the decrease of Futaki invariant of Fano 
varieties case directly, which we have just generalised as 
the previous Proposition \ref{Fut.decrease}.

\begin{Prop}[${\it dim}(B)>1$ case]\label{DF.decrease}

If $(\mathcal{X},\mathcal{L})$ minimises the Donaldson-Futaki invariants 
among other (birational) completions of $(\mathcal{X}^{o},\mathcal{L}^{o})
\twoheadrightarrow B^{o}$ satiefies following basic properties. 

\begin{enumerate}

\item $\mathcal{X}$ has only canonical singularities (we assume $\mathcal{X}$ is 
$\mathbb{Q}$-Gorenstein) if $(\mathcal{L}^{{\it dim}(F)}.K_{F})\le 0$ (e.g. 
if $-K_{F}$ is nef). 

\item If $K_{F}=a\mathcal{L}|_{F}$, then for a open neighborhood $U$ 
of $p$ in $B$, $K_{\mathcal{X}/B}|_{\pi^{-1}(U)}=
a\mathcal{L}|_{\pi^{-1}(U)}$. 

\end{enumerate}

\end{Prop}

\begin{proof}

(i): Suppose the contrary and take the relative canonical model of $\mathcal{X}$ 
as $f\colon \mathcal{X}^{\it can}\to \mathcal{X}$ by \cite{BCHM} again. 
Putting $E:=-K_{\mathcal{X}/B}$, we know that $E$ is effective by the negativity lemma 
again \cite[3.39]{KM}. Then, we have 
$$\frac{1}{n(n-1)} \frac{d}{dt}|_{t=0}{\it DF}(\mathcal{X}^{\it can},f^*\mathcal{L}-tE)$$ 

$$=-(f^{*}\mathcal{L}^{n-2}.E.f^{*}K_{\mathcal{X}/B}-E)(\mathcal{L}|_{F}^{{\it dim}(F)})^{\frac{n-1}{{\it dim}(F)}}$$

$$
+(f^{*}\mathcal{L}^{n-1}.E)(\mathcal{L}|_{F}^{{\it dim}(F)})^{\frac{{\it dim}(B)-1}{{\it dim}(F)}}
(\mathcal{L}|_{F}^{{\it dim}(F)-1}.K_{F}). 
$$

The first term is negative from the Hodge index theorem and the second is also 
negative due to our assumption. Thus we get a contradiction. 

(ii): It simply follows from Theorem \ref{Fut.decrease}. 
\end{proof}


The following, about the case over curve, are basically similar to the 
author's older works. The essential 
difference with the above higher dimensional base case is that we allow base change 
and consider ``normalised Futaki'' invariants concerning the degree of base change. 

Our modified setting is as follows. 
We fix $(\mathcal{X}^{o},\mathcal{L}^{o})\twoheadrightarrow B^{o}$ 
($=B\setminus\{p\}$ where $B$ is a smooth curve) 
as before, and consider all finite morphisms $\phi\colon \tilde{B}\to B$ and 
the normal $\mathbb{Q}$-Gorenstein completions of $(\mathcal{X}^{o},\mathcal{L}^{o})\times _{B}\tilde{B} \to 
\tilde{B}\setminus \{f^{-1}(p)\}$ to whole $\tilde{B}$ which we write 
$(\tilde{\mathcal{X}},\tilde{\mathcal{L}})$ and assume $\mathcal{X}$ is 
normal and $\mathbb{Q}$-Gorenstein. For this model, 
we set 
$$
{\it nDF}(\tilde{\mathcal{X}},\tilde{\mathcal{L}}):=\frac{{\it DF}(\tilde{\mathcal{X}},\tilde{\mathcal{L}})}
{{\it deg}(\phi)}, 
$$
as in \cite{LX}. 
We call $(\tilde{\mathcal{X}},\tilde{\mathcal{L}})$ is \textit{nDF-minimising} 
if ${\it nDF}(\tilde{\mathcal{X}},\tilde{\mathcal{L}})\le {\it nDF}(\tilde{\mathcal{X}}',\tilde{\mathcal{L}}')$ for all other possible $(\tilde{\mathcal{X}}',\tilde{\mathcal{L}}')$ 
(we allow all base changes $\phi$). 

\begin{Thm}[${\it dim}(B)=1$ case]
In the above setting, suppose $(\mathcal{X},\mathcal{L})$ is nDF-minimising. 
Then it holds that 
\begin{enumerate}
\item all the fibres are reduced and semi-log-canonical. 
\item if $F$ is a klt $\mathbb{Q}$-Fano variety, then all fibers are 
$\mathbb{Q}$-Fano varieties. 
\item if $K_{F}=a\mathcal{L}|_{F}$ with $a\ge 0$, the normalised DF invariant of 
$(\mathcal{X}',\mathcal{L}')\to B'$ is minimum all the models if and only if 
any fibre $G$ is  reduced slc with $K_{G}=a\mathcal{L}|_{G}$. 
\end{enumerate}
\end{Thm}

\begin{proof}

(i): If the fiber over $p\in B$ is not reduced, 
then we take base change of $B$ ramifying at $p$ with sufficiently divisible 
ramifying degree, and take its normalization. Then it is 
nontrivial along that preimage of non-reduced component so that the (normalized) 
DF decreases (cf., [remark after \ref{DF.pullback}], also \cite[5.1, 5.2]{RT}). 

So we can and do suppose $\mathcal{X}_{0}$ is reduced. 
The relative lc model of $(\mathcal{X},\mathcal{X}_{0})$ over $\mathcal{X}$ 
exists as \cite{OX} shows, which we denote as $f\colon \mathcal{X}^{{\it lc}}\to \mathcal{X}$. 

Then we consider a subtraction of a little of $E=K_{\mathcal{X}^{{\it lc}}}-f^{*}K_{\mathcal{X}}
+{\it Exc}(f)$ where ${\it Exc}(f)$ denotes the total reduced exceptional divisor i.e. 
$(\mathcal{X}^{{\it lc}},f^{*}\mathcal{L}(-tE))$ then the corresponding DF decrease if we change 
$t=0$ to $0<t\ll 1$, by the completely similar arguments as \cite[section 3]{Od0}. 

(ii): 
Supposing $\mathcal{X}_{0}$ is \textit{not} klt, 
then we take the relative log canonical model of $(\mathcal{X},(1-\epsilon)\mathcal{X}_{0})$ 
by \cite{BCHM}. Then the rest of the arguments is completely similar to \cite[section 6]{Od0}. 

(iii): ``If" direction is proved earlier in \cite[last section]{Od2} (also \cite{WX} for $a>0$ case). 
The converse holds as well, since if $\mathcal{L}-aK_{\mathcal{X}/B}$ is not $0$, 
then by Theorem \ref{Fut.decrease}, we have 
${\it DF}(\mathcal{L}+\epsilon E)<{\it DF}(\mathcal{L})$. 
\end{proof}

For Calabi-Yau case, a non-rigorous comment is that 
reduced slc CY fibers of $p\in B$ or its preimages by base changes 
form \textit{infinite ordered set} which we expect to  ``converges'' to 
tropical varieties which is homeomorphic to the dual complexes, as in  \cite{KS}. 
Regarding the Fano case (ii) we note that, by applying \cite{Kal}, 
it easily follows that all the completed $\mathbb{Q}$-Fano fibrations form 
a tree (as a graph) via Sarkisov links \cite{Kal}. We thank Annesophie Kaloghiros for 
answering the question regarding this. 

\subsubsection{Optimal destabilization}

Finally we propose a problem regarding ``maximal destablisation". 
We still keep the notation in section \ref{Gen.set}. 

\begin{Prob}
For fixed $\pi^{o}$, 
formulate the ``norms" $||\mathcal{L}||$ of polarisations $\mathcal{L}$ and 
show the existence of $(\mathcal{X},\mathcal{L})$ (``maximally destabilising'') which minimises the DF divided by the norm ${\it DF}(\mathcal{L})/||\mathcal{L}||$ 
among all birational models 
$\pi\colon(\mathcal{X},\mathcal{L})\twoheadrightarrow B$. 
\end{Prob}

We expect that relative canonical model $\mathcal{X}:=B^{{\it can}}\to B$ (cf. \cite{BCHM}, 
also recall the proof of (\ref{desing.ss})) will be the maximally destabilising model 
of non-(semi-)lc singularities in the case where $\mathcal{X}^{o}\simeq B^{o}$ 
(so that $\pi$ are birational). 
We also note that then the corresponding DF invariant is $(K_{\mathcal{X}/B})^{n}$ 
and its log version also exists. Indeed, take relative log canonical model 
of non-log-canonical base $B$ by \cite{OX} as $\mathcal{X}:=B^{{\it lc}}\to B$. 
Then the corresponding log DF invariant is $(K_{B^{{\it lc}}/B}+E)^{n}$ 
with reduced total exceptional divisor $E$ on $B^{{\it lc}}$, which coincides with 
the ``\textit{local volume}" of \cite{BdFF}, \cite{Ful} (for the proof of the 
coincidence, please see \cite{Zha}). 

\textbf{Acknowledgements}
I would like to take this opportunity to show my sincere gratitudes to 
Professors Futaki and Mabuchi for their pioneering works as well as  their precious and 
encouraging supports. 
They have been always very 
modest, and very kind to introduce me to 
analytic background of this area in various ways, 
since when I was a poor graduate student. 

For this paper, I deeply thank S.~Boucksom for helpful discussions, 
hospitality during my stay at the Institut de Jussieu from where 
one of the inspirations came and helpful comments on the exposition of the draft. 
I also thank K.~Fujita, S.~Mori, S.~Sun for the helpful comments on the draft. 

The author is partially supported by JSPS Kakenhi $\# 30700356$ 
(Wakate (B)). 










\pagebreak

\end{document}